\documentclass[11pt,reqno]{amsart}
\usepackage{amssymb,latexsym}
\usepackage{graphicx}
\usepackage{mathtools}
\usepackage{color}
\usepackage[T1]{fontenc}
\usepackage{hyperref}
\usepackage{amsmath}
\usepackage{breqn}
\usepackage[toc,page]{appendix}
\usepackage{url}    %does nice formatting of URLs
\newcommand{\bburl}[1]{\textcolor{blue}{\url{#1}}}

\calclayout

\makeatletter
\newcommand{\monthyear}[1]{%
  \def\@monthyear{\uppercase{#1}}}
\newcommand{\volnumber}[1]{%
  \def\@volnumber{\uppercase{#1}}}
\AtBeginDocument{%
\def\ps@plain{\ps@empty
  \def\@oddfoot{\@monthyear \hfil \thepage}%
  \def\@evenfoot{\thepage \hfil \@volnumber}}
\def\ps@firstpage{\ps@plain}
\def\ps@headings{\ps@empty
  \def\@evenhead{%
    \setTrue{runhead}%
    \def\thanks{\protect\thanks@warning}%
    \uppercase{\ }\hfil}%
  \def\@oddhead{%
    \setTrue{runhead}%
    \def\thanks{\protect\thanks@warning}%
    \hfill\uppercase{Gaussian Behavior in Zeckendorf Decompositions From Lattices}}%
  \let\@mkboth\markboth
  \def\@evenfoot{%
    \thepage \hfil \@volnumber}%
  \def\@oddfoot{%
    \@monthyear \hfil \thepage}%
  }%
\footskip=25pt
\pagestyle{headings}%
}
\makeatother

\theoremstyle{plain}
\numberwithin{equation}{section}
\newtheorem{thm}{Theorem}[section]

\newcommand{\ignore}[1]{}

\newcommand\be{\begin{eqnarray}}
\newcommand\ee{\end{eqnarray}}
\newcommand\bea{\begin{eqnarray}}
\newcommand\eea{\end{eqnarray}}
\newcommand\ben{\begin{enumerate}}
\newcommand\een{\end{enumerate}}

\newtheorem{lem}[thm]{Lemma}

\newtheorem{defi}[thm]{Definition}

\newcommand{\R}{\ensuremath{\mathbb{R}}}

\newcommand{\N}{\mathbb{N}}

\begin{document}
%% replace the values in the next three lines by the correct information
\monthyear{Month Year}
\volnumber{Volume, Number}
\setcounter{page}{1}
\title{Gaps of Summands of the Zeckendorf Lattice}

\author{Neelima Borade}
\address{Department of Mathematics, University of Illinois at Chicago, Chicago, IL 60607}
\email{nborad2@uic.edu}

\author{Dexter Cai}
\address{Wuhan Britain China School, Wansongyuan Road 48, Jianghan District, Wuhan, China}
\email{3182935374@qq.com}

\author{David Z. Chang}
\address{High Technology High School, Lincroft, NJ 07738}
\email{davidchang7636@gmail.com, dachang@ctemc.org}

\author{Bruce Fang}
\address{UWC Changshu China, No. 88 Kun-Cheng-Hu-Xi Road, Changshu, Jiangsu, China 215500}
\email{bjfang18@uwcchina.org}

\author{Alex Liang}
\address{Philips Exeter Academy, Exeter, NH 03833}
\email{yliang@exeter.edu}

\author{Steven J. Miller}
\address{Department of Mathematics and Statistics, Williams College, Williamstown, MA 01267}
\email{sjm1@williams.edu}

%%%\author{Heather Pan}
%%%\address{}
%%%\email{}

%%%\author{Yan (Roger) Weng}  %%%LATE
%%%\address{Peddie School,201 S Main St, Hightstown, NJ 08520}
%%%\email{roger_weng2018@outlook.com}

\author{Wanqiao Xu}
\address{Department of Mathematics, University of Michigan, Ann Arbor, MI 48109}
\email{wanqiaox@umich.edu}

\thanks{This work was supported by NSF Grant DMS1659037, Williams College and the Eureka Program. We thank Joshua Siktar for helpful conversations throughout the project.}

\begin{abstract}
A beautiful theorem of Zeckendorf states that every positive integer has a unique decomposition as a sum of non-adjacent Fibonacci numbers. Such decompositions exist more generally, and much is known about them. First, for any positive linear recurrence $\{G_n\}$ the number of summands in the legal decompositions for integers in $[G_n, G_{n+1})$ converges to a Gaussian distribution. Second, Bower, Insoft, Li, Miller, and Tosteson proved that the probability of a gap between summands in a decomposition which is larger than the recurrence length converges to geometric decay. While most of the literature involves one-dimensional sequences, some recent work by Chen, Guo, Jiang, Miller, Siktar, and Yu have extended these decompositions to $d$-dimensional lattices, where a legal decomposition is a chain of points such that one moves in all $d$ dimensions to get from one point to the next. They proved that some but not all properties from 1-dimensional sequences still hold. We continue this work and look at the distribution of gaps between terms of legal decompositions, and prove similar to the 1-dimensional cases that when $d=2$ the gap vectors converge to a bivariate geometric random variable.
\end{abstract}



\maketitle

\section{Introduction}
\label{sec:introduction}

%%%%%%%%%%%%%%%%%%%%%%%%%%%%%%%%%%%%%%%%%%%%%%%%%%%%%%%%%%%%%%%%%%%
%%%%%%%%%%%%%%%%%%%%%%%%%%%%%%%%%%%%%%%%%%%%%%%%%%%%%%%%%%%%%%%%%%%
%%%%%%%%%%%%%%%%% Past Literature: TO BE ADDED %%%%%%%%%%%%%%%%%%%%
%%%%%%%%%%%%%%%%%%%%%%%%%%%%%%%%%%%%%%%%%%%%%%%%%%%%%%%%%%%%%%%%%%%
%%%%%%%%%%%%%%%%%%%%%%%%%%%%%%%%%%%%%%%%%%%%%%%%%%%%%%%%%%%%%%%%%%%

%%%%%%%%%%%%%%%%%%%%%%%%%%%%%%%%%%%%%%%%%%%%%%%%%%%%%%%
%%%%%%%%%%%%%%%%%%%%%%%%%%%%%%%%%%%%%%%%%%%%%%%%%%%%%%%
%%%%%%%%%%%%%%%%%%%%%%%%%%%%%%%%%%%%%%%%%%%%%%%%%%%%%%%
\subsection{Previous Work}

Zeckendorf's Theorem \cite{Ze} provides an alternative definition of the Fibonacci numbers $\{F_n\}$ (normally defined by $F_1=1, F_2=2$ and $F_{n+1} = F_n + F_{n-1}$ for all $n\geq 2$): they are the only sequence such that every positive integer can be written uniquely as sum of non-adjacent terms. Such a sum is called the Zeckendorf (or legal) decomposition. Similar results hold for other sequences; see for example \cite{Bes,Br,Day,Dem,FGNPT,Fr,GTNP,Ha,Ho,HW,Ke,Lek,LM1,LM2,Mw1,Mw2,Ste1,Ste2} for a representative sample of results on unique decompositions, as well as on the distribution of the number of summands in these decompositions. Most of the work to date has been on one-dimensional sequences; many of the sequences which at first appear two-dimensional, such as those in \cite{CFHMN2,CFHMNPX}, are in fact one-dimensional when viewed properly. In \cite{CCGJMSY}, the authors considered generalizations to $d$-dimensional lattices, where a legal decomposition involved a finite subset of lattice points where each point has all of its coordinates strictly smaller than the previous (thus all motion is both down and to the left). The motivation for their work was to see which properties persist. They were able to show that the distribution of the number of summands is similar, also a Gaussian, but unique decomposition is lost.

In this work we continue these investigations and look at the distribution of gaps between summands in decompositions. For many sequences the resulting distributions converge to geometric decay, with the constant arising from the largest root of the characteristic polynomial of the recurrence relation; see \cite{BBGILMT,Bow}. The question is more interesting here, as the gaps are now $d$-dimensional vectors, and there is the possibility of new behavior.

We first recall the construction from \cite{CCGJMSY}; the following description is slightly modified from that work with permission of the authors. A legal decomposition in $d$ dimensions is a finite collection of lattice points for which
\begin{enumerate}
\item each point is used at most once, and
\item if the point $(i_1, i_2, \dots, i_d)$ is included then all subsequent points $(i_1', i_2', \dots, i_d')$ have $i'_j < i_j$ for all $j \in \{1, 2, \dots, d\}$ (i.e., \emph{all} coordinates must decrease between any two points in the decomposition).
\end{enumerate}

We call these sequences of points on the $d$-dimensional lattice simple jump paths. One can weaken the second condition and merely require $i'_j \le i_j$. This restriction was imposed in \cite{CCGJMSY} to simplify the combinatorial analysis, as it led to simple closed form expressions. With additional work we can consider the more general case where it is no longer required that all coordinates decrease; see \cite{FJLLLMSS}, where the authors prove similar behavior as in \cite{CCGJMSY}.

We now construct our sequence using the above definition of legal decomposition. We concentrate on $d \in \{1, 2\}$ in the main results in this paper; similar results should hold in general, but for small $d$ there are combinatorial identities which simplify the sums that arise and lead to nice closed form expression. Whenever possible we state definitions and ancillary lemmas for the most general case possible.

When $d=1$ we write $\lbrace y_a\rbrace^{\infty}_{a = 0}$ for our sequence, which is defined by

\begin{enumerate}

\item $y_1 := 1$, and

\item if we have constructed the first $k$ terms of our sequence, the $(k+1)$\textsuperscript{st} term is the smallest integer which cannot be written as a sum of terms in the sequence, with each term used at most once.

\end{enumerate}

This case is, not surprisingly, very similar to previous one-dimensional results. A straightforward calculation shows that $a_n = 2^{n-1}$, and the legal decomposition of a number is just its binary expansion.

We now turn to the main object of study in this paper, $d=2$, and denote the general term of our sequence by $y_{i,j}$. Instead of defining the sequence by traveling along diagonals we could do right angular paths; for the purposes of this paper it does not matter, as we are only concerned with the gaps between chosen summands, and not the values of the summands (and it is the values which are influenced by the mode of construction).

\begin{enumerate}
\item Set $y_{1,1} := 1$.

\item Iterate through the natural numbers. For each such number, check if any path of numbers in our sequence with a strict leftward and downward movement between each two points sums to the number. If no such path exists, add the number to the sequence so that it is added to the shortest unfilled diagonal moving from the bottom right to the top left.

\item If a new diagonal must begin to accommodate a new number, set the value $y_{k, 1}$ to be that number, where $k$ is minimized so that $y_{k, 1}$ has not yet been assigned.
\end{enumerate}

In \eqref{ZeckendorfDiagonalSequenceSimp2D} we give the first few diagonals of the $2$-dimensional lattice. Note that we no longer have uniqueness of decomposition (for example, $25$ has two legal decompositions: $20+5$ and $24+1$).
\begin{eqnarray}
\begin{array}{cccccccccc}280 & \cdots & \cdots & \cdots & \cdots & \cdots & \cdots & \cdots & \cdots & \cdots \\157 & 263 & \cdots & \cdots & \cdots & \cdots & \cdots & \cdots & \cdots & \cdots \\84 & 155 & 259 & \cdots & \cdots & \cdots & \cdots & \cdots & \cdots & \cdots \\50 & 82 & 139 & 230 & \cdots & \cdots & \cdots & \cdots & \cdots & \cdots \\28 & 48 & 74 & 123 & 198 & \cdots & \cdots & \cdots & \cdots & \cdots \\14 & 24 & 40 & 66 & 107 & 184 & \cdots & \cdots & \cdots & \cdots \\7 & 12 & 20 & 33 & 59 & 100 & 171 & \cdots & \cdots & \cdots \\3 & 5 & 9 & 17 & 30 & 56 & 93 & 160 & \cdots & \cdots \\1 & 2 & 4 & 8 & 16 & 29 & 54 & 90 & 159 & \cdots \end{array}
\label{ZeckendorfDiagonalSequenceSimp2D}
\end{eqnarray}

The main result of \cite{CCGJMSY} is that the distribution of the number of summands among all simple jump paths starting at $(n,n)$ and ending at $(0,0)$ converges to a Gaussian as $n\to\infty$ (as all paths must have both a down and a left component, we can add an additional row and an additional column where one of the indices is zero, and require all paths to end at $(0,0)$); a similar result holds for compound paths where each step is either down, left or down and left \cite{FJLLLMSS}. We investigate the distribution of gaps between adjacent summands in legal decompositions. Before stating our results, we first introduce some notation.

%%%%%%%%%%%%%%%%%%%%%%%%%%%%%%%%%%%%%%%%%%%%%%%%%%%%%%%
%%%%%%%%%%%%%%%%%%%%%%%%%%%%%%%%%%%%%%%%%%%%%%%%%%%%%%%
%%%%%%%%%%%%%%%%%%%%%%%%%%%%%%%%%%%%%%%%%%%%%%%%%%%%%%%
\subsection{New Results}

There are several ways to define gaps in these $d$-dimensional lattice decompositions, leading to slightly different behavior. We give three possibilities here.

\begin{defi}\label{defi:gap-vector}
    For a step from $(x_{m, 1}, \dots, x_{m, d})$ to $(x_{m+1, 1}, \dots, x_{m+1, d})$, its \textbf{gap vector} is the difference $(x_{m,1}-x_{m+1,1}, \dots, x_{m,d}-x_{m+1,d})$. A simple jump path of length $k$ starting at $(a_1,a_2,\dots,a_d)$ corresponds to the set $\{(x_{i,1}-x_{i+1,1},\dots,x_{i,d}-x_{i+1,d})\}_{i=0}^{k-1}$ of $k$ gap vectors, where
\begin{itemize}
    \item $(x_{0, 1}, \dots, x_{0, d}) = (a_1, \dots, a_d)$,

    \item $(x_{k, 1}, \dots, x_{k, d}) = (0, \dots, 0)$, and

    \item for each $i \in \{0,1, \dots, k - 1\}$ and $j \in \{1, \dots, d\}$, $x_{i, j} > x_{i + 1, j}$.
\end{itemize}
\end{defi}

\begin{defi}\label{defi:gap-sum}
   Given a gap vector $(x_{i,1}-x_{i+1,1},\dots,x_{i,d}-x_{i+1,d})$, its \textbf{gap sum} is the sum of the components of the vector: $(x_{i,1}-x_{i+1,1})+\cdots+(x_{i,d}-x_{i+1,d})$. Similarly the \textbf{gap distance} is the length of the gap vector.
\end{defi}

There are three natural quantities we can investigate. We can look at the gap vectors, the gap sums, or the gap distances. The distribution of the gap vectors is the most fundamental quantity, and much of the combinatorics is a natural generalization of previous work for the one-dimensional case \cite{BBGILMT,Bow}. Knowing the distribution of the gap vectors, we can calculate the distribution of the gap sums by summing over all gap vectors with the same gap sum. The last notion, the gap distance, is harder as this requires summing over a subset of gap vectors to obtain a given gap distance. Note we can interpret the difference between these two perspectives as arising from the norm we use to measure the length of the gap vector; the gap sum comes from using the $L^1$ norm while the gap distance is from the $L^2$ norm.

Our main result is that as $n$ goes to infinity, the distribution of the gap vectors in the $2$-dimensional lattice converges to a geometric decay, and thus we see similar behavior as in the one-dimensional cases.

\begin{thm}\label{thm:intro-gap-vector-dist}
    Let $n$ be a positive integer, consider the distribution of gap vectors among all simple jump paths of dimension $2$ with starting point $(n + 1,n + 1)$. For fixed positive integers $v_1, v_2$, the probability that a gap vector equals $(v_1,v_2)$ converges to $1/2^{v_1+v_2}$ as $n\to\infty$.
\end{thm}

We prove Theorem \ref{thm:intro-gap-vector-dist} in \S \ref{sec:gapsin2-dimensionallattice} through combinatorial identities and Stirling's formula, but for larger $d$ the combinatorial lemmas do not generalize. As an immediate consequence, we obtain the distribution of the gap sums.

\begin{thm}\label{thm:intro-probabilityofagapdef2}
    Let $n$ be a positive integer, consider the distribution of gap sums among all simple jump paths of dimension $2$ with starting point $(n+1,n+1)$. The probability that a gap sum equals an integer $v \ge 2$ converges to $(v-1)/2^v$ as $n\to\infty$ (the probability of a gap sum of 0 or 1 is zero).
\end{thm}

After reviewing properties of simple jump paths in \S\ref{sec: propertiesofsimplepaths}, we prove our main results in \S\ref{sec:gapsin2-dimensionallattice}, and then conclude the paper with questions on alternative definitions of gaps, generalizations to compound paths, and distribution of the longest gap (where we present some partial results). Whenever possible we state and prove results for arbitrary dimensions, to facilitate future research.

%%%%%%%%%%%%%%%%%%%%%%%%%%%%%%%%%%%%%%%%%%%%%%%%%%%%%%%%%%%%%%%%%%%%%%%%%%%%%%%%%%%%%%%%%%%%%%%%%%%%%%%%%%%%%%%%%%%%%%%%%%%%%%%%%%
%%%%%%%%%%%%%%%%%%%%%%%%%%%%%%%%%%%%%%%%%%%%%%%%%%%%%%%%%%%%%%%%%%%%%%%%%%%%%%%%%%%%%%%%%%%%%%%%%%%%%%%%%%%%%%%%%%%%%%%%%%%%%%%%%%
%%%%%%%%%%%%%%%%%%%%%%%%%%%%%%%%%%%%%%%%%%%%%%%%%%%%%%%%%%%%%%%%%%%%%%%%%%%%%%%%%%%%%%%%%%%%%%%%%%%%%%%%%%%%%%%%%%%%%%%%%%%%%%%%%%
\section{Properties of Simple Jump Paths}\label{sec: propertiesofsimplepaths}
We first recall some notation for our simple jump paths from \cite{CCGJMSY}. As our paper is an extension of \cite{CCGJMSY},
the following four paragraphs are reproduced with permission from them.

We have walks in $d$ dimensions starting at some initial point $(a_1, a_2, \dots, a_d)$ with each $a_j  > 0$, and ending at the origin $(0, 0, \dots, 0)$. Note that our simple jump paths must always have movement in all dimensions at each step. We are just adding one extra point, at the origin, and saying every path must end there. Note that as we always change all of the indices during a step, we never include a point where only some of the coordinates are zero, and thus there is no issue in adding one extra point and requiring all paths to end at the origin.

Our walks are sequences of points on the lattice grid with positive indices or the origin, and we refer to movements between two such consecutive points as \textbf{steps}. Thus a simple jump path is a walk where each step has a strict movement in all $d$-dimensions. More formally, a simple jump path of length $k$ starting at $(a_1, a_2, \dots, a_d)$ is a sequence of points $\{(x_{i, 1}, \dots, x_{i, d})\}^{k}_{i = 0}$ where the following hold:
\begin{itemize}
\item $(x_{0, 1}, \dots, x_{0, d}) = (a_1, \dots, a_d)$,

\item $(x_{k, 1}, \dots, x_{k, d}) = (0, \dots, 0)$, and

\item for each $i \in \{0,1, \dots, k - 1\}$ and $j \in \{1, \dots, d\}$, $x_{i, j} > x_{i + 1, j}$.

\end{itemize}

For a fixed $d$ and any choice of starting point $(a_1, a_2, \dots, a_d) \in \R^d$, we let $s_d( a_1,\dots,a_d)$ denote the number of simple jump paths starting at $(a_1, a_2, \dots, a_d)$ and ending at $(0, \dots, 0)$, and $t_d(k; a_1,\dots,a_d)$ the subset of these paths with length $k$. In particular, when $a_1=\cdots = a_d = n$ for a fixed $n \in \mathbb{N}^+$, we let $s_d(n)$ denote the number of simple jump paths from $(n, n, \dots, n)$ to the origin, and $t_d(k, n)$ the subset of these paths with exactly $k$ steps. As we must reach the origin, every path has at least 1 step, the maximum number of steps is $n$, and
\begin{equation} \label{simpleJumpPathPartitionByNumSteps}
s_d(n) \ = \ \sum_{k = 1}^{n} t_d(k, n).
\end{equation}

We now determine $t_d(k, n)$. In one dimension we have $t_d(k, n) = \binom{n-1}{k-1}$, as we must choose exactly $k-1$ of the first $n-1$ terms (we must choose the $n^{\text{th}}$ term as well as the origin, and thus choosing $k-1$ additional places ensures there are exactly $k$ steps). Since in higher dimension there is movement in each dimension for each step, for $1 \leq k \leq \min(a_1,\dots, a_d)$,
\begin{eqnarray}\label{eq:t_d}
t_d(k; a_1, \dots, a_d) \ = \ \binom{a_1-1}{k-1}\binom{a_2-1}{k-1}\cdots\binom{a_d-1}{k-1},
\end{eqnarray} and
\begin{eqnarray}\label{eq:s_d}
s_d(a_1, \dots, a_d) \ = \ \sum_{k=1}^{\min(a_1, \dots, a_d)} t_d(k;a_1, \dots, a_d).
\end{eqnarray}

From the binomial theorem we have $s_1(a_1) = 2^{a_1-1}$. For higher dimensions we need another well-known combinatorial result: Vandermonde's identity. We reproduce the proof here for completeness.

\begin{lem}[Vandermonde's Identity]\label{lem:binomial identity}
For $\alpha,\beta\in\N$,
\begin{equation}
\sum_{k=0}^{\gamma}\binom{\alpha}{\gamma-k}\binom{\beta}{k} \ = \ \binom{\alpha+\beta}{\gamma}.
\end{equation}

\begin{proof}
Suppose there are $\alpha+\beta$ balls, with $\alpha$ of them red and $\beta$ of them blue. We want to select $\gamma$ balls in total; there are $\binom{\alpha+\beta}{\gamma}$ ways to do this. We can also calculate this by splitting into cases of the number of blue balls chosen. There are $\binom{\alpha}{\gamma-k}\binom{\beta}{k}$ ways to choose $k$ blue balls and $\gamma-k$ red balls. Summing over all cases, it's clear that
\begin{equation}\label{eq:numberofsimplepaths}
\sum_{k=0}^{\gamma}\binom{\alpha}{\gamma-k}\binom{\beta}{k} \ = \ \binom{\alpha+\beta}{\gamma},
\end{equation}
which completes the proof.
\end{proof}
\end{lem}

We can now determine the number of simple paths in $2$ dimensions. The result below is an extension of results from \cite{CCGJMSY}, where only the special case $a_1 = a_2$ has been proved.

\begin{thm}\label{thm:numberofsimplepaths}
In 2-dimensional lattice,
\begin{equation}
s_2(a_1,a_2) \ = \ \binom{a_1+a_2-2}{a_1 - 1}.
\end{equation}

\begin{proof}
From \eqref{eq:t_d} and \eqref{eq:s_d}, we have
\begin{equation}
s_2(a_1,a_2) \ = \ \sum_{k=1}^{\min(a_1,a_2)} \binom{a_1-1}{k-1} \binom{a_2-1}{k-1}.
\end{equation}
Without loss of generality, assume $\min(a_1,a_2) = a_1$. Then
\begin{eqnarray}
s_2(a_1,a_2) & \ = \ & \sum_{k=1}^{a_1} \binom{a_1-1}{k-1} \binom{a_2-1}{k-1}\nonumber\\
& \ = \ & \sum_{k=0}^{a_1-1} \binom{a_1-1}{a_1-1-k} \binom{a_2-1}{k}.
\end{eqnarray}
Applying Lemma \ref{lem:binomial identity},
\begin{equation}\label{eq:generalizationofnumberofsimplepaths}
s_2(a_1,a_2) \ = \ \binom{a_1+a_2-2}{a_1-1} \ = \ \binom{a_1+a_2-2}{a_2-1}.
\end{equation}
\end{proof}
\end{thm}

Note that when $a_1 = a_2 = n$, we have
\begin{equation}
s_2(n) = \binom{2n-2}{n-1}.
\end{equation}

%%%%%%%%%%%%%%%%%%%%%%%%%%%%%%%%%%%%%%%%%%%%%%%%%%%%%%%%%%%%%%%%%%%%%%%%%%%%%%%%%%%%%%%%%%%%%%%%%%%%%%%%%%%%%%%%%%%%%%%%%%%%%%%%%%
%%%%%%%%%%%%%%%%%%%%%%%%%%%%%%%%%%%%%%%%%%%%%%%%%%%%%%%%%%%%%%%%%%%%%%%%%%%%%%%%%%%%%%%%%%%%%%%%%%%%%%%%%%%%%%%%%%%%%%%%%%%%%%%%%%
%%%%%%%%%%%%%%%%%%%%%%%%%%%%%%%%%%%%%%%%%%%%%%%%%%%%%%%%%%%%%%%%%%%%%%%%%%%%%%%%%%%%%%%%%%%%%%%%%%%%%%%%%%%%%%%%%%%%%%%%%%%%%%%%%%

\section{Gaps in 2-dimensional Lattice}\label{sec:gapsin2-dimensionallattice}

In $1$-dimension case, the notion of gaps between adjacent points in a simple jump path is unambiguous; however, for $d\geq 2$, we have several choices. Below we concentrate on the gap vector.

%For a step from $(x_{m, 1}, \dots, x_{m, d})$ to $(x_{m+1, 1}, \dots, x_{m+1, d})$, we refer to the difference $(x_{m,1}-x_{m+1,1}, \dots, x_{m,d}-x_{m+1,d})$ as a \textbf{gap vector}. For instance, a simple jump path of length $k$ starting at $(a_1,a_2,\dots,a_d)$ corresponds to the set $\{(x_{i,1}-x_{i+1,1},\dots,x_{i,d}-x_{i+1,d})\}_{i=0}^{k-1}$ of $k$ gap vectors, where
%\begin{itemize}
 %   \item $(x_{0, 1}, \dots, x_{0, d}) \ = \ (a_1, \dots, a_d)$,
%
 %   \item $(x_{k, 1}, \dots, x_{k, d}) \ = \ (0, \dots, 0)$, and
%
 %   \item for each $i \in \{1, \dots, k - 1\}$ and $j \in \{1, \dots, d\}$, $x_{i, j} \ > \ x_{i + 1, j}$.
%\end{itemize}

We let $g_d(n)$ denote the number of gap vectors of all simple jump paths from $(n,n,\dots,n)$ to the origin. We add the origin to the path to facilitate the counting; as all steps must be both down and to the left, the origin is the only point where one of the indices is zero. While this does introduce one extra gap, as $n\to\infty$ the contribution from it is negligible and can thus be safely ignored. This addition now means that each simple jump path of length $k$ contains $k$ gap vectors, and every legal path has at least one and at most $n$ gap vectors. Thus
\begin{equation}\label{simpleJumpPathPartitionByNumGaps}
    g_d(n) \ = \ \sum_{k=1}^n k \ t_d(k,n).
\end{equation}

%When $d\geq 2$, for a gap vector $(x_{i,1}-x_{i+1,1},\dots,x_{i,d}-x_{i+1,d})$, we define a \textbf{gap sum} to be $(x_{i,1}-x_{i+1,1})+\cdots+(x_{i,d}-x_{i+1,d})$, the sum of the components of the vector.

To prove Theorems \ref{thm:intro-gap-vector-dist} and \ref{thm:intro-probabilityofagapdef2}, we begin with determining $g\left(n;(v_1,v_2)\right)$, defined as the number of gap vectors $(v_1,v_2)$ in all simple jump paths starting from $(n,n)$ and ending at $(0,0)$. Then we find $g_2(n)$, the total number of gap vectors. Due to the presence of $n-1$ in the formula for $s_2(n)$, we work with $n+1$ below to simplify some of the algebra.

\begin{lem}\label{lem:numberofgapstartingatapoint}
Consider all the simple jump paths from $(n+1,n+1)$ to $(0,0)$ in the $2$-dimensional lattice. Let $G\left((x,y),(x+v_1,y+v_2)\right)$ denote the number of gap vectors $(v_1,v_2)$ starting at $(x+v_1,y+v_2)$ and ending at $(x,y)$ within all simple jump paths from $(n+1,n+1)$ to $(0,0)$, then
\begin{equation}
G\left((x,y), (x+v_1,y+v_2)\right) \ = \ \binom{x+y-2}{x-1}\binom{2n-v_1-v_2-x-y}{n-v_1-x}.
\end{equation}

\begin{proof}
Since each different arrangement of simple jump paths from $(n+1,n+1)$ to $(x+v_1,y+v_2)$ and from $(x,y)$ to $(0,0)$ contributes $1$ to the number of gap vectors $(v_1,v_2)$, $G\left((x,y), (x+v_1,y+v_2)\right)$ is given by the number of simple jump paths from $(x,y)$ to $(0,0)$ times the number of simple paths from $(n+1,n+1)$ to $(x+v_1,y+v_2)$; see Figure \ref{fig:gridgapbig}. 

\begin{figure}
\begin{center}
\scalebox{.6}{\includegraphics{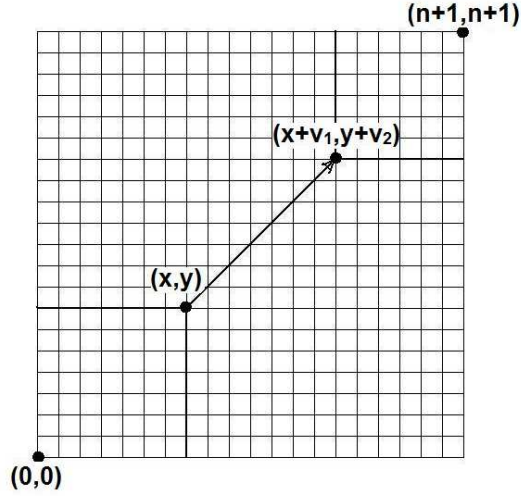}}
\caption{\label{fig:gridgapbig} Set-up to compute the number of simple jump paths from $(0,0)$ to $(n+1,n+1)$ with a gap of $(v_1,v_2)$ starting at $(x,y)$. Counting the number of such paths is the same as counting the number of paths in the bottom left rectangle and multiplying by the number of paths in the top right.}
\end{center}\end{figure}

Shifting the coordinates $(x+v_1,y+v_2)$ and $(n+1,n+1)$ down to $(0,0)$ and $(n-x-v_1+1,n-y-v_2+1)$ respectively, we obtain
\begin{equation}
G\left((x,y),(x+v_1,y+v_2)\right) \ = \ s_2\left(x,y\right)\cdot s_2\left(n-x-v_1+1,n-y-v_2+1\right).
\end{equation}
Applying Theorem \ref{thm:numberofsimplepaths},
\begin{equation}
G\left((x,y),(x+v_1,y+v_2)\right) \ = \ \binom{x+y-2}{x-1}\binom{2n-v_1-v_2-x-y}{n-v_1-x}.
\end{equation}
\end{proof}
\end{lem}

Now we determine the range of $x,y$ in Lemma \ref{lem:numberofgapstartingatapoint}. If a gap vector $(v_1,v_2)$ starts at $(x+v_1,y+v_2)$ and ends at $(x,y)$, then it is clear that $x,y\geq 0$. Since we are only considering simple jump paths from $(n+1,n+1)$ to $(0,0)$, the components of $(x+v_1,y+v_2)$ cannot exceed $n+1$. Thus $x \le n-v_1+1$ and $y \le n-v_2+1$. Combining, we have $0\le x \le n-v_1+1$ and $0\le y \le n-v_2+1$.

\begin{lem}\label{numberofagap}
Recall that $g\left(n+1;(v_1,v_2)\right)$ denotes the number of gap vectors $(v_1,v_2)$ in all simple jump paths starting from $(n+1,n+1)$ and ending at $(0,0)$. Then
\begin{equation}
g\left(n+1;(v_1,v_2)\right)\ =\ (2n-v_1-v_2-1)\binom{2n-v_1-v_2-2}{n-v_1-1}+2\binom{2n-v_1-v_2}{n-v_1}.
\end{equation}

\begin{proof}
For ease of calculation, we break into three cases:
\begin{enumerate}
\item $1\le x\le n-v_1$ and $1\le y\le n-v_2$,
\item $x=0$ and $y=0$,
\item $x=n-v_1+1$ and $y=n-v_2+1$.
\end{enumerate}
We first consider Case (1). By Lemma \ref{lem:numberofgapstartingatapoint}, the number of gap vectors $(v_1,v_2)$ is given by
\begin{equation}\label{eq:case1}
\sum_{x=1}^{n-v_1}\sum_{y=1}^{n-v_2}G\left((x,y),(x+v_1,y+v_2)\right) \ = \ \sum_{x=1}^{n-v_1}\sum_{y=1}^{n-v_2}\binom{x+y-2}{x-1}\binom{2n-v_1-v_2-x-y}{n-v_1-x}.
\end{equation}
Shifting the index of $x,y$ in the sum, the right hand side of \eqref{eq:case1} becomes
\begin{equation}\label{eq:case1simplify}
\sum_{x=0}^{n-v_1-1}\sum_{y=0}^{n-v_2-1}\binom{x+y}{x}\binom{2n-v_1-v_2-x-y-2}{n-v_1-x-1}.
\end{equation}
Letting $p=n-v_1-1$ and $q=n-v_2-1$, it is equivalent to calculate
\begin{equation}
\sum_{x=0}^{p} \sum_{y=0}^{q}\binom{x+y}{x}\binom{p+q-(x+y)}{p-x}.
\end{equation}
In the sum, since $0\le x\le p$ and $0\le y\le q$,
\begin{equation}
0\ \le\ x+y\ \le\ p+q,
\end{equation}
so there are $p+q+1$ different values of $x+y$. Now we bring the problem down to proving that, for each fixed value of $x+y$,
\begin{equation}\label{eq:fixing_x+y}
\sum_{x=0}^{p}\binom{x+y}{x}\binom{p+q-(x+y)}{p-x} \ = \ \binom{p+q}{p},
\end{equation}
which is an immediate application of lemma \ref{lem:binomial identity}. Note that not all terms in the left hand side of \eqref{eq:fixing_x+y} are necessarily non-zero, as $x$ cannot exceed $x+y$.

Thus the number of gap vectors $(v_1,v_2)$ in Case (1) is
\begin{equation}
(p+q+1)\binom{p+q}{p} \ = \  (2n-v_1-v_2-1)\binom{2n-v_1-v_2-2}{n-v_1-1}.
\end{equation}
We then consider Case (2). When $x=0$, $y=0$, the number of gap vectors $(v_1,v_2)$ equals the number of simple jump paths from $(n+1,n+1)$ to $(v_1,v_2)$. Shifting the coordinates $(n+1,n+1)$ and $(v_1,v_2)$ down to $(n+1-v_1,n+1-v_2)$ and $(0,0)$, respectively, the number of gap vectors in this case $(v_1,v_2)$ is just $s_2(n+1-v_1,n+1-v_2)$.
Applying Theorem \ref{thm:numberofsimplepaths},
\begin{equation}
s_2(n+1-v_1,n+1-v_2) \ = \ \binom{2n-v_1-v_2}{n-v_1}.
\end{equation}
Similarly, in Case (3), the number of gap vectors $(v_1,v_2)$ is
\begin{equation}
s_2(n+1-v_1,n+1-v_2) \ = \ \binom{2n-v_1-v_2}{n-v_1}.
\end{equation}
Summing up all three cases,
\begin{equation}
g\left(n+1;(v_1,v_2)\right) \ = \ (2n-v_1-v_2-1)\binom{2n-v_1-v_2-2}{n-v_1-1} +2 \binom{2n-v_1-v_2}{n-v_1}.
\end{equation}
\end{proof}
\end{lem}

%%%%%%%%%%%%%%%%%%%%%%%%%%%%%%%%%%%%%%%%%%%%%%%%%%%%%%%%%%%%
\begin{lem}\label{lem:totalnumberofgaps}
\begin{equation}
g_2\left(n+1\right) \ = \ \left(\frac{n}{2}+1\right)\binom{2n}{n}.
\end{equation}
\end{lem}
We prove the lemma using two different methods, in anticipation that perhaps one might be more useful to future researchers trying to generalize to higher dimensions.

\begin{proof}[First Proof]
Let $\mu_2\left(n+1\right)$ be the mean for the number of steps of all simple jump paths from $(n+1,n+1)$ to $(0,0)$, and $s_2(n+1)$ be the total number of simple paths from $(n+1,n+1)$ to $(0,0)$. By Lemma $3.1$ in \cite{CCGJMSY},
\begin{equation}
\mu_2\left(n+1\right) \ = \ \frac{n}{2}+1,
\end{equation}
and
\begin{equation}
s_2(n+1) \ = \ \binom{2n}{n};
\end{equation} it is here that we are using $d=2$, as it is only when $d\le 2$ that we have simple formulas for $s_d(n+1)$, though with a more involved analysis similar results shoudl be obtainable for all $d$.

Since in a simple jump path the number of steps is the same as the number of gap vectors,
\begin{equation}
g_2\left(n+1\right) \ = \ \left(\frac{n}{2}+1\right)\binom{2n}{n}.
\end{equation}
\end{proof}

\begin{proof}[Second Proof]
Since a simple jump path from $(n+1,n+1)$ to $(0,0)$ with $i$ points chosen in between (not counting $(n+1,n+1)$ and $(0,0)$) has $(i+1)$ gap vectors,
\begin{eqnarray}
g_2\left(n+1\right) & \ = \ & \sum_{i=0}^{n}(i+1)\binom{n}{i}\binom{n}{i}\nonumber\\
& \ = \ & \sum_{i=0}^{n}i\binom{n}{i}+\sum_{i=0}^{n}\binom{n}{i}\binom{n}{i}.
\end{eqnarray}
We use the following two standard binomial identities, see for example  \cite{CCGJMSY} (again we do not have these for general $d$):
\begin{eqnarray}
\sum_{i=0}^{n}\binom{n}{i}\binom{n}{i} &\ = \ & \binom{2n}{n} \nonumber\\
\sum_{i=0}^{n}i\binom{n}{i}\binom{n}{i} &\ = \ & \frac{n}{2}\binom{2n}{n}.
\end{eqnarray}
Hence
\begin{eqnarray}
g_2\left(n+1\right) & \ = \ & \frac{n}{2}\binom{2n}{n} + \binom{2n}{n}\nonumber\\
& \ = \ & \left(\frac{n}{2}+1\right)\binom{2n}{n}.
\end{eqnarray}
\end{proof}

%%%%%%%%%%%%%%%%%%%%%%%%%%%%%%%%%%%%%%%%%%%%%%%%%%%%%%%%%
Now we have all the tools to prove Theorem \ref{thm:intro-gap-vector-dist}.

\begin{proof}[Proof of Theorem \ref{thm:intro-gap-vector-dist}]
Let $P(v_1,v_2)$ denote the probability that a given gap vector is $(v_1,v_2)$. By definition,
\begin{eqnarray}\label{eqn:probabilityofagapequationdisplay}
P(v_1,v_2) & = & \frac{G\left(n+1;(v_1,v_2)\right)}{g_2\left(n+1\right)}.
\end{eqnarray}
Using Lemmas \ref{numberofagap} and \ref{lem:totalnumberofgaps} to simplify,
\begin{eqnarray}\label{eqn:probabilityofagapequationfactorial}
P(v_1,v_2) & = &  \frac{(2n-v_1-v_2-1)\binom{2n-v_1-v_2-2}{n-v_1-1}+2\binom{2n-v_1-v_2}{n-v_1}}{(\frac{n}{2}+1)\binom{2n}{n}} \nonumber\\ & = & \frac{\frac{(2n-v_1-v_2-1)!}{(n-v_1-1)!(n-v_2-1)!}+2\frac{(2n-v_1-v_2)}{(n-v_1)(n-v_2)}\frac{(2n-v_1-v_2-1)!}{(n-v_1-1)!(n-v_2-1)!}}{(\frac{n}{2}+1)
\frac{(2n)!}{n!n!}}\nonumber\\ & = &  \frac{\left(1+2\frac{(2n-v_1-v_2)}{(n-v_1)(n-v_2)}\right)\frac{(2n-v_1-v_2-1)!}{(n-v_1-1)!(n-v_2-1)!}}{(\frac{n}{2}+1)
\frac{(2n)!}{n!n!}}.
\end{eqnarray}
As $v_1$ and $v_2$ are fixed,
\begin{equation}
\lim_{n\rightarrow\infty} \frac{(2n-v_1-v_2)}{(n-v_1)(n-v_2)}\ =\ 0,
\end{equation}
and thus
\begin{equation}
\lim_{n\to\infty}P(v_1,v_2) \ =\ \lim_{n\rightarrow\infty}\frac{\frac{(2n-v_1-v_2-1)!}{(n-v_1-1)!(n-v_2-1)!}}{(\frac{n}{2}+1)\frac{(2n)!}{n!n!}}.
\end{equation}

For $u$ large, Stirling's approximation states that $m! \approx \sqrt{2\pi m}(\frac{m}{e})^{m}$. We approximate the factorials, and can safely drop the lower order error terms as we take the limit as $n\to\infty$ and obtain
\begin{eqnarray}
\lim_{n\to \infty}P(v_1,v_2) & \ =\ & \lim_{n\rightarrow\infty}\frac{\left(\frac{2n-v_1-v_2-1}{e}\right)^{2n-v_1-v_2-1}\left(\frac{n}{e}\right)^{2n}}{\left(\frac{n-v_1-1}{e}\right)^{n-v_1-1} \left(\frac{n-v_2-1}{e}\right)^{n-v_2-1} \left(\frac{n}{2}+1\right)\left(\frac{2n}{e}\right)^{2n}}\nonumber\\ & & \ \ \ \ \times\frac{(\sqrt{2\pi})^{3}\sqrt{2n-v_1-v_2-1}(\sqrt{n})^{2}}{(\sqrt{2\pi})^{3}\sqrt{n-v_1-1}\sqrt{n-v_2-1}\sqrt{2n}}.\label{eq:limit-of-prob}
\end{eqnarray}
Since
\begin{equation}
\lim_{n\rightarrow\infty}\frac{\sqrt{2n-v_1-v_2-1}(\sqrt{n})^{2}}{\sqrt{n-v_1-1}\sqrt{n-v_2-1}\sqrt{2n}}\ =\ 1,
\end{equation}
The right hand side of \eqref{eq:limit-of-prob} becomes
\begin{eqnarray}
\lim_{n\to\infty} P(v_1,v_2) &\ =\ & \lim_{n\rightarrow\infty} \frac{\left(\frac{2n-v_1-v_2-1}{e}\right)^{2n-v_1-v_2-1}\left(\frac{n}{e}\right)^{2n}}{\left(\frac{n-v_1-1}{e}\right)^{n-v_1-1}\left(\frac{n-v_2-1}{e}\right)^{n-v_2-1}\left(\frac{n}{2}+1\right)\left(\frac{2n}{e}\right)^{2n}}\nonumber\\ & = & \lim_{n\rightarrow\infty}\frac{e^{-1}(2n-v_1-v_2-1)^{2n-v_1-v_2-1}}{\left(\frac{n}{2}+1\right)2^{2n}(n-v_1-1)^{n-v_1-1}(n-v_2-1)^{n-v_2-1}}\nonumber\\ & = & \lim_{n\rightarrow\infty}\frac{2n-v_1-v_2-1}{\frac{n}{2}+1}\frac{e^{-1}}{2^{v_1+v_2+2}}\left(\frac{2n-v_1-v_2-1}{2n-2v_1-2}\right)^{n-v_1-1}\nonumber\\ & & \ \ \ \ \times\left(\frac{2n-v_1-v_2-1}{2n-2v_2-2}\right)^{n-v_2-1}\nonumber\\ & = &\lim_{n\rightarrow\infty}\frac{2n-v_1-v_2-1}{\frac{n}{2}+1}\frac{e^{-1}}{2^{v_1+v_2+2}}\left(1+
\frac{\frac{v_1-v_2+1}{2}}{n-v_1-1}\right)^{n-v_1-1}\nonumber\\ & & \ \ \ \ \times \left(1+\frac{\frac{v_2-v_1+1}{2}}{n-v_2-1}\right)^{n-v_2-1}.
\end{eqnarray}
As
\begin{equation}
\lim_{x\rightarrow\infty}\left(1+\frac{a}{x}\right)^{x} \ = \ e^{a},
\end{equation}
we find
\begin{eqnarray}
\lim_{n\to\infty}P(v_1,v_2) & \ =\ & \lim_{n\rightarrow\infty}\frac{2n-v_1-v_2-1}{\frac{n}{2}+1}\frac{e^{-1}}{2^{v_1+v_2+2}}e^{\frac{v_1-v_2+1}{2}}e^{\frac{v_2-v_1+1}{2}} \nonumber\\ & = &\lim_{n\rightarrow\infty} \frac{2n-v_1-v_2-1}{\frac{n}{2}+1}\frac{1}{2^{v_1+v_2+2}}.
\end{eqnarray}

Since
\begin{equation}
\lim_{n\to\infty}\frac{2n-v_1-v_2-1}{\frac{n}{2}+1}\ =\ 4,
\end{equation}
we obtain
\begin{eqnarray}
\lim_{n\to\infty}P(v_1,v_2) \ =\ 4\times\frac{1}{2^{v_1+v_2+2}}\ = \  \frac{1}{2^{v_1+v_2}},
\end{eqnarray}
which is clearly a bivariate geometric random variable.
\end{proof}

We have proved that the distribution of gap vectors converge to a geometric decay as the lattice size goes to infinity. We may now turn to an alternate definition of gaps, namely, the gap sum. Recall Definition \ref{defi:gap-sum}, which states that a gap sum is the sum of components of the corresponding gap vector. The proof for Theorem \ref{thm:intro-probabilityofagapdef2} follows immediately from Theorem \ref{thm:intro-gap-vector-dist}.

\begin{proof}[Proof of Theorem \ref{thm:intro-probabilityofagapdef2}]
For a fixed $n$, let $P(v)$ denote the probability that a given gap sum equals $v \ge 2$; note $P(0) = P(1) = 0$ as we must have both horizontal and vertical movement in a step, therefore the smallest possible gap sum is 2. By Theorem \ref{thm:intro-gap-vector-dist}, for each value $v$ of gap sum, all pairs $(v_1,v_2)$ with $v_1 + v_2 = v$ contribute equally. As $1 \le v_1, v_2 \le v-1$, there are $v-1$ such pairs (once $v_1$ is chosen then $v_2$ is determined), each pair occurring with probability $1/2^v$. Thus
\begin{equation}
\lim_{n\to\infty}P(v) \ = \ (v-1)\left(\frac{1}{2}\right)^{v},
\end{equation} completing the proof.
\end{proof}

We remark on the difficulty in generalizing the above argument to arbitrary $d$. The problem is Lemma \ref{lem:binomial identity}; we are not aware of an analogue when $d\geq 3$.

%%%%%%%%%%%%%%%%%%%%%%%%%%%%%%%%%%%%%%%%%%%%%%%%%%%%%%%%%%%%%%%%%%%%%%%%%%%%%%%%%%%%%%%%%%%%%%%%%%%%%%%%%%%%%%%%%%%%%%%%%%%%%%%%%%
%%%%%%%%%%%%%%%%%%%%%%%%%%%%%%%%%%%%%%%%%%%%%%%%%%%%%%%%%%%%%%%%%%%%%%%%%%%%%%%%%%%%%%%%%%%%%%%%%%%%%%%%%%%%%%%%%%%%%%%%%%%%%%%%%%
%%%%%%%%%%%%%%%%%%%%%%%%%%%%%%%%%%%%%%%%%%%%%%%%%%%%%%%%%%%%%%%%%%%%%%%%%%%%%%%%%%%%%%%%%%%%%%%%%%%%%%%%%%%%%%%%%%%%%%%%%%%%%%%%%%
\section{Future Work and Concluding Remarks}\label{sec:futureworkandconcludingremarks}

We end with some problems and comments for future research.

\begin{enumerate}
\item Is there a way to generalize our analysis to the $d$-dimensional lattice?
\item Do nice limits for the distribution of gap distances exist as they do for gap vectors and gap sums?
\item Can we obtain similar results in $d$-dimensional compound path \cite{FJLLLMSS} with the three definitions of gaps we set forth in this paper?
\item Can we obtain similar results on the distribution of longest gap in both $d$-dimensional simple paths and compound paths?
\end{enumerate}

As \cite{CCGJMSY} was able to obtain Gaussian behavior for the distribution of summands for all $d$, there is reason to be optimistic that a more involved analysis is possible and we could obtain similar extensions for gaps. In that work, however, the simple closed form expressions that exist in two dimensions do not generalize, and combinatorial proofs and analysis are replaced by more involved techniques. We have thus chosen here to concentrate on $d \le 2$, as this is already enough to see new behavior (i.e., previous problems never looked at the distribution of vector valued quantities for gaps).

For the distribution of gap distances, we need to know not just what numbers are the sums of two squares, but how they are, as the probability of a gap vector $(v_1, v_2)$ is $1/2^{v_1+v_2}$. Thus for small $g$ we can easily compute all pairs $(v_1, v_2)$ with $v_1^2 + v_2^2 = g$ (it is easier to study the square of the gap distance, as that is always an integer). For large $g$ we would need advanced results from number theory on decompositions, but these values will have negligibly small probabilities of occurring.

As \cite{FJLLLMSS} extended the results from \cite{CCGJMSY} to compound paths through additional book-keeping (especially more involved inclusion-exclusion arguments), with more work it is likely that the gap results can be generalized to the compound setting as well.

Finally, we end with some results on the distribution of the longest gap. The one dimensional case is already known (see for example \cite{GSW,Sch}); it is essentially equivalent to the distribution of the longest run of heads when tossing a fair coin. If we toss $n$ fair coins, the expected value of the longest run of heads is \be  \log_{2} n + \frac{\gamma}{\log 2} - \frac32  + r_1(n) + \epsilon_1(n), \ee where $\gamma$ is Euler's constant,  $|r_1(n)| \le  0.000016$, and $\epsilon_1(n)$ tends to zero as $n$ tends to infinity. Moreover, the distribution is strongly concentrated about the mean; the variance is \be \frac{\pi^2}{6\log^2 2} + \frac{1}{12} + r_2(n) + \epsilon_2(n), \ee where $r_2(n) < 0.00006$ and $\epsilon_2(n)$ goes to zero as $n$ tends to infinity. Note the mean is approximately $\log_2 n$ and the variance is \emph{bounded} independent of $n$.

Interestingly, it is easier to get results in the compound setting, as the freedom to just move in one direction allows us to view the two components of the vectors as independent. In other words, if we wish to look at the length of the longest horizontal or vertical gap, it is essentially the same as in the case of tossing fair coins; the only possible difference is we must end with a `head', but at worst that increases the length by 1, which is negligible relative to $\log_2 n$. We can thus immediately get decent bounds on the approximate size of the longest gap in the compound case; the horizontal and vertical results give a lower bound, and adding the two (or adding the two and taking a square-root) provides an upper bound. More work is needed to get results for simple paths, because there now \emph{is} a dependence between the two motions (we must have the same number of `heads' for each), but with some work it is likely that one could obtain results that show $\log_2 n$ is the right order of magnitude for the longest gap (or at least that there is negligible probability of a longest gap of size $n^\delta$ for any fixed $\delta > 0$, and that the longest gap is of size $\log\log n$).

%%%%%%%%%%%%%%%%%%%%%%%%%%%%%%%%%%%%%%%%%%%%%%%%%%%%%%%%%%%%%%%%%%%%%%%%%%%%%%%%%%%%%%%%%%%%%%%%%%%%%%%%%%%%%%%%%%%%%%%%%%%%%%
%%%%%%%%%%%%%%%%%%%%%%%%%%%%%%%%%%%%%%%%%%%%%%%%%%%%%%%%%%%%%%%%%%%%%%%%%%%%%%%%%%%%%%%%%%%%%%%%%%%%%%%%%%%%%%%%%%%%%%%%%%%%%%
%%%%%%%%%%%%%%%%%%%%%%%%%%%%%%%%%%%%%%%%%%%%%%%%%%%%%%%%%%%%%%%%%%%%%%%%%%%%%%%%%%%%%%%%%%%%%%%%%%%%%%%%%%%%%%%%%%%%%%%%%%%%%%
%%%%%%%%%%%%%%%%%%%%%%%%%%%%%%%%%%%%%%%%%%%%%%%%%%%%%%%%%%%%%%%%%%%%%%%%%%%%%%%%%%%%%%%%%%%%%%%%%%%%%%%%%%%%%%%%%%%%%%%%%%%%%%

\medskip

\noindent MSC2010: 11B02 (primary), 05A02 (secondary)

\end{document}